\documentclass[11pt,reqno]{amsart}
\usepackage{fullpage}
\usepackage[mathscr]{eucal}
\usepackage{mathrsfs}
\usepackage{amsfonts}
\usepackage{amsmath}
\usepackage{amsthm}
\usepackage{amssymb}
\usepackage{latexsym}
\usepackage[all]{xy}
\usepackage{array}

\usepackage{enumitem}

\usepackage[dvipsnames]{xcolor}
\usepackage{graphicx}

\newcommand{\red}[1]{{\color{red}#1}}

\usepackage{mathtools}
\usepackage{latexsym}
\usepackage[mathscr]{eucal}
\usepackage{palatino}
\usepackage[bookmarks=false]{hyperref}
\usepackage{hyperref}
\usepackage{xr-hyper}

\usepackage{graphicx}

\theoremstyle{plain}
\newtheorem{thm}[equation]{Theorem}
\newtheorem{cor}[equation]{Corollary}
\newtheorem{prop}[equation]{Proposition}
\newtheorem{lem}[equation]{Lemma}

\theoremstyle{definition}

\newtheorem*{defn}{Definition}

\theoremstyle{remark}
\newtheorem{quest}[equation]{Question}

\newtheorem{rem}[equation]{Remark}
\newtheorem{rems}[equation]{Remarks}

\numberwithin{equation}{section}
\allowdisplaybreaks

\newlength{\dhatheight}

\newcommand{\x}[1]{\xrightarrow{#1}}
\newcommand{\hdot}{{\:\raisebox{2pt}{\text{\circle*{1.5}}}}}
\newcommand{\idot}{{\:\raisebox{2pt}{\text{\circle*{1.5}}}}}

\newcommand{\rst}[1]{\ensuremath{{\mathbin|}\raise-.5ex\hbox{$#1$}}}  
\DeclareMathOperator{\Map}{\mathrm{Map}}

\DeclareMathOperator{\Lie}{\mathrm{Lie}}

\DeclareMathOperator{\ad}{\mathrm{ad}}

\DeclareMathOperator{\Spec}{\mathrm{Spec}}

\newcommand{\iso}{{\;\stackrel{_\sim}{\to}\;}}

\newcommand{\erem}{\hfill$\lozenge$\end{rem}}
\newcommand{\erems}{\hfill$\lozenge$\end{rems}}
\newcommand{\dis}{\displaystyle}

\newcommand{\beq}{\begin{equation}\label}
\newcommand{\eeq}{\end{equation}}
\def\ccirc{{{}_{\,{}^{^\circ}}}}
\DeclareMathOperator{\Hom}{\mathrm{Hom}}

\newcommand{\pt}{{\op{pt}}}

\newcommand{\BG}{{\mathbb G}}

\newcommand{\Q}{{\mathbb Q}}

\newcommand{\Bun}{{\textsl{Bun}}}

\newcommand{\X}{{\Sigma}}

\newcommand{\higgs}{{\textsl{Higgs}}}

\newcommand{\aff}{^{\textrm{aff}}}

\newcommand{\cl}{{\mathcal L}}

\renewcommand{\o}{\otimes }

\newcommand{\ccong}{\ \cong \  }

\newcommand{\wt}{\widetilde }

\newcommand{\Om}{\Omega }

\newcommand{\Id}{{\operatorname{Id}}}

\newcommand{\XX}{{\mathcal X}}
\newcommand{\Y}{{\mathcal Y}}

\def\ccirc{{{}_{^{\,^\circ}}}}

\newcommand{\la}{\lambda}

\newcommand{\op}{\operatorname}

\newcommand{\oo}{{\mathcal{O}}}

\newcommand{\GG}{{G\times\BG_m}}

\newcommand{\eu}{{\op{eu}}}

\renewcommand{\sec}{{\mathcal{S}^{\!}\textit{ect}}}
\newcommand{\qsec}{{\mathcal{QS}^{\!}\textit{ect}}}
\newcommand{\qsecd}{{\mathcal{QS}^{\!}\textit{ect}}_{d}(M_L/G)}
\newcommand{\cz}{{\mathcal Z}}

\newcommand{\ev}{{\textrm{ev}}}

\newcommand{\C}{\mathbb{C}}
\newcommand{\g}{\mathfrak{g}}

\newcommand{\n}{\mathfrak{n}}
\newcommand{\m}{\mathfrak{m}}
\newcommand{\h}{\mathfrak{h}}

\newcommand{\mm}{{\mathcal M}}

\newcommand{\inv}{^{-1}}

\newcommand{\Z}{{\mathbb Z}}
\newcommand{\en}{{\enspace}}

\newcommand{\vi}{${\en\sf {(i)}}\;$}
\newcommand{\vii}{${\;\sf {(ii)}}\;$}
\newcommand{\viii}{${\sf {(iii)}}\;$}

\newcommand{\sset}{\subset}
\newcommand{\sminus}{\smallsetminus}
\newcommand{\intoo}{\,\xymatrix{\ar@{^{(}->}[r]&}\,}
\newcommand{\ontoo}{\,\xymatrix{\ar@{->>}[r]&}\,}
\newcommand{\into}{\,\hookrightarrow\,}
\newcommand{\too}{\,\longrightarrow\,}
\newcommand{\mto}{\mapsto}

\newcommand{\N}{{\mathcal{N}}}

\newcommand{\om}{\omega }

\newcommand{\on}{\operatorname}
\usepackage{lipsum}

\newcommand\blfootnote[1]{%
  \begingroup
  \renewcommand\thefootnote{}\footnote{#1}%
  \addtocounter{footnote}{-1}%
  \endgroup
}

\begin{document}
%\hfill\textsf{Preliminary draft, not for distribution.}\break
\title{\mbox{Gaiotto's Lagrangian subvarieties via derived symplectic geometry}}
\author{Victor Ginzburg}
\address{
Department of Mathematics, University of Chicago,  Chicago, IL 
60637, USA.}
\email{\ ginzburg@math.uchicago.edu,\qquad nick@math.uchicago.edu}
\author{Nick Rozenblyum}
%\address{N.R.:
%\Department of Mathematics, University of Chicago,  Chicago, IL 
%\60637, USA.}
%\email{nick@math.uchicago.edu}

\maketitle
\vskip 10pt
\hfill{\em To Alexander Alexandrovich Kirillov with gratitude and admiration}\break

\begin{abstract} Let $\Bun_G$ be the moduli space of $G$-bundles on a smooth complex projective curve. Motivated by a study of boundary conditions in mirror symmetry,
D. Gaiotto \cite{Ga} associated to any symplectic representation of $G$ a Lagrangian subvariety of $T^*Bun_G$.
We give a simple interpretation of  (a generalization of) Gaiotto's construction in terms of derived symplectic geometry. This allows to consider
a more general setting where  symplectic $G$-representations are replaced by arbitrary symplectic manifolds equipped with
a Hamiltonian $G$-action and with an action of the multiplicative group that rescales the symplectic form with positive weight.
\end{abstract}
{\small
\tableofcontents
}
%sss

\section{Statement of the result}\blfootnote{The authors  are grateful to Davide Gaiotto, Kevin Costello,
and Li Yu for
inspiring discussions.
The first author  was supported in part by the NSF grant DMS-1303462.}
We will use the language of derived  stacks. Throughout, a `stack' means a `derived Artin stack  over $k=\C$' in the sense of \cite{GR} and \cite{PTVV}.
We write $B\BG=\pt/\BG$ for the classifying stack of a group ~$\BG$.
We fix  a smooth complex projective variety $X$ and
let  $K_ X$ denote the canonical bundle. We write $G$ for  an algebraic group
and  $\Bun_G( X)$, resp. $\higgs_G(X)$, for the stack of $G$-bundles,  resp. Higgs bundles,
on $X$. 
 One has a canonical  isomorphism $\Bun_G(X)\cong\Map( X, BG)$,
where $\Map(X, Z)$ denotes a   mapping stack that classifies morphisms $ X\to Z$.

Given a $\BG_m$-stack $\Y$ and a  $\BG_m$-bundle $L\to  X$,
there is an associated bundle
$\Y_L:=\Y\times_{\BG_m}L$. Let   $\sec_X(\Y_L)$ be the stack  of sections of the  projection $\Y_L\to X$. By definition,
we have $\sec_X(\Y_L)=\{\Id_X\}\,\times_{\Map(X,X)}\,\Map(X,\Y_L)$.
The  
$\BG_m$-action on the first factor   of  $\Y\times L$
descends to a $\BG_m$-action along the fibers of $\Y_L\to X$.
This induces a natural
 $\BG_m$-action on $\sec_X(\Y_L)$.

\begin{rems} Let  $L\to X$ be a $\BG_m$-bundle and  $\cl$ an associated line bundle on $X$.

\vi We will abuse the notation and write $Y_\cl$ for $Y_L$.

\vii For a $\BG_m$-stack $\Y$, there is a canonical isomorphism
$\Y_L\cong \Y/\BG_m \,\times_{_{B\BG_m}}\, X$, where we have used
the map $X=L/\BG_m\to B\BG_m=\pt/\BG_m$ 
that  classifies $L$. 

\viii For a $(G\times \BG_m)$-stack  $\Y$,
we will often use natural identifications
${(\Y/G)_L=(\Y\times L)/(G\times \BG_m)}$ $=(\Y_L)/G$.
\erems

Let $M$ be a smooth symplectic algebraic manifold
equipped with a $G\times \BG_m$-action such that
the action of the group $G=G\times\{1\}$ on $M$ is
Hamiltonian and the symplectic 2-form  has  weight $\ell\geq1$ with respect to the action
of $\BG_m=\{1\}\times \BG_m$. 
Assume that there exists a line bundle $K_ X^{1/\ell}$, an $\ell$-th root of $K_X$,
and fix a choice of  $K_ X^{1/\ell}$. 

Following  Gaiotto, \cite{Ga},
we consider  the stack 
$\sec_X(M_{K_ X^{1/\ell}}/G)$. This stack classifies pairs $(P,s)$, where $P$ is a $(\GG)$-bundle on $X$ and
$s: P \to M\times {\overset{_{\circ}}K}_ X^{_{1/\ell}}$ is a $(G\times \BG_m)$-equivariant morphism
that intertwines the natural projections $P\to X$ and  $M\times {\overset{_{\circ}}K}_ X^{_{1/\ell}}\to X$.
Here ${\overset{_{\circ}}K}_ X^{_{1/\ell}}$ denotes the $\BG_m$-bundle obtained from $K_ X^{1/\ell}$ by removing the zero section.
The group $G$ acts on $M\times {\overset{_{\circ}}K}_ X^{_{1/\ell}}$   through its action on the first factor
and $\BG_m$ acts  diagonally.

Let   $\g$ be the Lie algebra of $G$ and $\g^*$ the dual of $\g$.
The group $G\times\BG_m$ acts on $\g^*$, where $G$ acts by the coadjoint action
and $\BG_m$ acts by dilations. 
The symplectic 2-form on $M$ being of weight $\ell$, the moment map  $\mu: M \to \g^* $ intertwines, for any $t\in \BG_m$,
the $t$-action  on $M$ with  dilation by $t^\ell$ on $\g^*$.
It follows that $\mu$ gives a well defined
morphism $M_{K_ X^{1/\ell}}\to \g^*_{K_X}$, of stacks over $X$. Therefore, there is an induced morphism
\beq{mumap} \mu_{\sec}:\ \sec_X( M_{K_ X^{1/\ell}}/G)\too \sec_X(\g^*_{K_X}/G).
\eeq

We now specialize to the case where $X=\X$ is a smooth projective curve
and $G$ is reductive. In such a case, we have $\sec_\X(\g^*_{K_\X}/G)\cong\higgs_G(X)\cong
T^*\Bun_G(\X)$.
Let $T^*\Bun_G(\X)^{reg}$ be 
an open substack of $T^*\Bun_G(\X)$ that corresponds to the
  Higgs bundles whose only automorphisms lie in the center. It is known that   $T^*\Bun_G(\X)^{reg}$ is a smooth  variety
that comes equipped with a natural symplectic 2-form $\om$.

\begin{thm}\label{c1} The map $\mu_{\sec}$ is Lagrangian, specifically, the 2-form $\mu_{\sec}^*(\om)$
vanishes on the preimage of  $T^*\Bun_G(\X)^{reg}$.
\end{thm}

The above result  was discovered by Gaiotto \cite{Ga} in the linear case, i.e. in the special case where $M$ is a symplectic representation
of $G$. In this case,  $\BG_m$  acts on $M$, a symplectic vector space, by dilations
and the symplectic form on $M$ has weight $2$.

One of the goals of this paper is to show that Theorem \ref{c1} 
is  a simple consequence of some very general results of derived symplectic geometry.

\section{Derived symplectic geometry} 
Let $n$ be an integer and  $ Y$  a stack equipped
with an $n$-shifted symplectic structure in the sense of \cite{PTVV}.
There is a  notion  of  ``Lagrangian structure" 
on a morphism $Z\to  Y$, see \cite[\S 2.2]{PTVV} and \cite{Ca}.
One has the following result, where part (i) is \cite[Theorem 0.4]{PTVV}, resp. part (ii) is 
 \cite[Therorem 2.10]{Ca}.

\begin{thm}\label{thm} Let $X$  be a smooth projective Calabi-Yau variety of dimension $d$. Then, one has:

\vi An  $n$-shifted symplectic structure on a stack $Y$ gives rise to a natural $(n-d)$-shifted symplectic structure
on $\Map(X, Y)$.

\vii A Lagrangian structure $f: Z\to  Y$  gives rise to a natural Lagrangian structure on $\Map(X,Z)\to \Map(X, Y)$, the
morphism of mapping stacks induced by $f$.
\end{thm}

It was shown, see  \cite[Corollary 2.6(2)]{PTVV},  that part (i) of the theorem implies the following

\begin{cor}\label{prop}  For any  smooth projective Calabi-Yau variety $X$ of dimension $d$
the stack $\higgs_G(X)$ has a canonical $2(1-d)$-shifted symplectic structure.
\end{cor}

In the case where $X$ is a Fano variety  suitable analogues of the statements of Theorem \ref{thm}
were proved  by Spaide \cite{derived2}, Theorem 3.3 and Theorem 3.5.

\medskip

Below, we propose a modification of the above results that holds  for more general, not necessarily
 Calabi-Yau, varieties $X$.

To this end, we recall some notions from derived algebraic geometry. For a (derived) stack
$\XX$, we will denote by $\on{QCoh}(\XX)$ the (unbouded) derived $\infty$-category of quasi-coherent sheaves
on $\XX$ (see, e.g. \cite{GR} for a detailed account of this $\infty$-category). We will refer to objects
of $\on{QCoh}(\XX)$ as ``sheaves on $\XX$.''  Given $\mathcal{M} \in \on{QCoh}(\XX)$, we will denote
by $\Gamma(\XX, \mathcal{M})=Hom(\mathcal{O}_\XX, \mathcal{M})$, the ({\em derived}) functor of global sections.

Let $f: Y \to \XX$ be a map of stacks and $\mathbb{L}_{Y/\XX} \in \on{QCoh}(Y)$ the relative cotangent
complex of $f$. One has a sheaf 
$$\tilde{\mathcal{A}}^p_\XX(Y) := f_*(\wedge^p \mathbb{L}_{Y/\XX})\in
\on{QCoh}(\XX), $$
of relative $p$-forms. There is  also a sheaf $\tilde{\mathcal{A}}^{p,cl}_\XX(Y) \in
\on{QCoh}(\XX)$, of  relative  closed  $p$-forms. The sheaf  $\tilde{\mathcal{A}}^{p,cl}_\XX(Y)$
comes equipped with
 a forgetful map $\tilde{\mathcal{A}}^{p,cl}_\XX(Y) \to \tilde{\mathcal{A}}^{p}_\XX(Y)$ which assigns
to a closed $p$-form its underlying $p$-form (see \cite[Sect. 1]{CPTVV} or \cite[Vol. II, Chapter 9]{GR} for a discussion of 
relative differential forms). 
Note that in the derived setting, a closed $p$-form is a $p$-form equipped with additional closure data (as opposed to satisfying a condition).

We will use the following basic result about relative differential forms:

\begin{lem}\label{p:rel forms}
Let
$$
\xymatrix{
Y_2 \ar[r]\ar[d] & Y_1\ar[d] \\
\XX_2 \ar[r]^g & \XX_1
}
$$
be a commutative square of stacks.  Then, for each $i\geq 0$,  there is a natural map
$$ \phi_{i,cl}: g^*(\tilde{\mathcal{A}}^{i,cl}_{\XX_1}(Y_1)) \to \tilde{\mathcal{A}}^{i,cl}_{\XX_2}(Y_2).$$
Moreover, if the square is Cartesian and $\mathbb{L}_{Y_1/\XX_1}$ is perfect
(more generally, it is sufficient to require  $\mathbb{L}_{Y_1/\XX_1}$  be  bounded below)
then the map $\phi_{p,cl}$ is an  isomorphism.
\end{lem}
\qed

\begin{defn} Let  $p: Y\to \XX$ be a map of stacks  and  $\mathcal{L}$ a line bundle on $\XX$. We put
$$ \mathcal{A}^i(Y/\XX;\mathcal{L}):=\Gamma(\XX, \tilde{\mathcal{A}}^i_\XX(Y) \otimes \mathcal{L}), \mbox{ and } \mathcal{A}^{i,cl}(Y/\XX; \mathcal{L}):=\Gamma(\XX, 
\tilde{\mathcal{A}}^{i,cl}_\XX(Y) \otimes \mathcal{L}) .$$

\vi Assume the relative cotangent complex of $p: Y\to \XX$ is perfect.  {\em An $\mathcal{L}$-twisted $n$-shifted relative symplectic structure on} $Y$ is a twisted relative closed 2-form $\omega \in Hom(k,\mathcal{A}^{2,cl}(Y/\XX; \mathcal{L})[n])$ such that the underlying 2-form is nondegenerate, i.e. it induces an isomorphism
$$ \mathbb{L}_{Y/\XX}^{\vee} \overset{\sim}{\to} \mathbb{L}_{Y/\XX}[n] \otimes p^*(\mathcal{L}). $$

\vii Assume that
 $p: Y\to \XX$ is equipped with an $\mathcal{L}$-twisted $n$-shifted relative symplectic structure and let $f: Z\to Y$
be a map of stacks with perfect relative cotangent complex.  {\em An ($\mathcal{L}$-twisted $n$-shifted) Lagrangian structure on} $f$ is a nullhomotopy of $f^*(\omega) \in Hom(k, \mathcal{A}_\XX^{2,cl}(Z; \mathcal{L})[n])$ such that the map
$$ \mathbb{L}_{Z/\XX}^{\vee} \to \mathbb{L}_{Z/Y}[n-1] \otimes (f\circ p)^*(\mathcal{L}) ,$$
induced by the nullhomotopy of the underlying 2-form, is an isomorphism.
%\end{enumerate}
\end{defn}

The proposition below gives a preliminary version of our main construction.
  In Section \ref{s:equiv}, we will describe how to obtain relative twisted symplectic, resp. Lagrangian, structures from symplectic,
resp. Lagrangian, sturctures of a fixed weight on a $\BG_m$-stack.

\begin{prop}\label{t:rel AKSZ}
Let $X$ be  a smooth projective variety  of dimension $d$ and $Y,Z$ a pair of stacks.
\begin{enumerate}[label=(\roman*)]
\item
A $K_X$-twisted relative $n$-shifted symplectic structure on $Y\to X$ induces an $(n-d)$-shifted symplectic structure on $\sec_X(Y)$.
\item
A $K_X$-twisted relative Langrangian structure on $Z\to Y$ induces a Lagrangian structure on
$$ \sec_X(Z)\to \sec_X(Y).$$
\end{enumerate}
\end{prop}
\begin{proof}
Following \cite{PTVV}, we consider the  evaluation map
$$ \sec_X(Y)\times X \overset{\ev}{\to} Y,$$
a map of stacks over $X$.  By Lemma \ref{p:rel forms}, there is a pull-back morphism in $\on{QCoh}(X)$:
$$ ev^*: \tilde{\mathcal{A}}^{2,cl}_X(Y) \otimes_{\mathcal{O}_X} K_X \to \mathcal{A}^{2,cl}(\sec_X(Y)) \otimes_k K_X. $$
Using an integration map
$\int_X: \Gamma(X, K_X) \to k[-d] $ provided by Serre duality, one
obtains a map
$$ (\Id\times\mbox{$\int_X$}) \circ\  ev^* : \mathcal{A}^{2,cl}(Y/X; K_X) \to \mathcal{A}^{2,cl}(\sec_X(Y)) .$$
 
Now, the same argument as in \cite{PTVV} shows that
if the twisted $2$-form $\om$ on $Y$ is nondegenerate then so is the $2$-form
$$\om_{\sec}:=(\Id\times\mbox{$\int_X$}) \  \circ\  ev^*)(\om).$$
This proves part (i) of Proposition \ref{t:rel AKSZ}. The proof of part (ii) is obtained by similarly tweaking the proof of \cite[Therorem 2.10]{Ca}.
\end{proof}

\begin{rems}\label{r:compact oriented stack}
\vi The same proof works in a more general setting where $X$ is any strictly $\mathcal{O}$-compact stack in the sense of \cite[Definition 2.1]{PTVV} 
equipped  with a line bundle $K_X$ and a map
$ \int_X: \Gamma(X,K_X) \to k[-d] $
that induces a perfect pairing as in \cite[Definition 2.4]{PTVV}.  For instance, one can take $X$ be any proper Gorenstein (derived) scheme.

\vii 
 It is tempting to try to develop a formalism of `derived hyper-K\"ahler geometry', at least a notion
of  `derived twistor space'. One could then consider an analogue of Proposition \ref{t:rel AKSZ}, as well as analogues of various results below,
 with a hyper-K\"ahler target $Y$ and hyper-Lagrangian
structures $Z\to Y$.
\end{rems}

\section{Equivariance and twistings}\label{s:equiv}
Let $Y$ be  a $\BG_m$-stack. 
Given an integer $m$, let $Y^{(m)}$ denote the $\BG_m$-stack
with the same underlying stack as $Y$ and the $\BG_m$-action given by precomposition with the homomorphism 
$\BG_m \to \BG_m ,\ t \mapsto t^m$.  The space of (closed) $p$-forms on the  $\BG_m$-stack $Y$ carries a
natural $\Z$-grading, to be referred to as  `weight'. Thus, one can consider $n$-shifted symplectic  structures on $Y$
of weight $m$. 

Given  a $\BG_m$-stack $Z$,
 we say that $f$ is a map  from $Z$ to $Y$ of weight $m$ if $f$ is a
$\BG_m$-equivariant
map $Z \to Y^{(m)}$.
Heuristically, a map  $f:Z\to Y$ has  weight $m$ if  $f(tz) = t^m f(z)$ for all $t\in \BG_m$.

\begin{defn} Fix an $n$-shifted symplectic  structure on $Y$
of weight $m$. This  gives, for each $\ell\geq1$,
an    $n$-shifted symplectic  structure on $Y^{(\ell)}$ of weight $m\ell$.

 \vi An equivariant  Lagrangian structure is an equivariant map $f: Z\to Y$, of $\BG_m$-stacks,
equipped with a nullhomotopy, {\em in the space of closed 2-forms on $Z$ of weight $m$},
of the pullback of the  $n$-shifted symplectic form, satisfying a non-degeneracy condition.

\vii 
An equivariant  Lagrangian structure  $f: Z\to Y^{(\ell)}$
will be called a  Lagrangian structure of weight $\ell$.
\end{defn}

%Moreover, suppose $Y$ is a $\BG_m$-stack and $L\to X$ is a
%$\BG_m$-torsor corresponding to a line bundle $\mathcal{L}$ on $X$.  Let
%$$ Y_L := Y \underset{\BG_m}{\times} L \simeq Y/\BG_m \underset{B\BG_m}{\times} X, $$
%where the map $X\to B\BG_m$ classifies $L$.  The stack $Y_L$ carries a fiberwise $\BG_m$-action over $X$ and induces a $\BG_m$-action on $\sec_X(Y_L)$.
Let $X$ be a smooth projective variety of dimension $d$ (or, more generally, a derived stack with a
twisted orientation of degree $d$ as in Remark \ref{r:compact oriented stack}).  
Fix $m\in \mathbb{Z}$ and  a choice, $K^{1/m}$, of an $m$-th root of the line bundle $K_X$ on $X$.

\begin{lem}\label{l:twisted rel sympl} Let
$Y$ be a  $\BG_m$-stack equipped with an $n$-shifted symplectic form
of weight $m\geq 1$ with respect to the $\BG_m$-action.  Let $\mathcal{L}$ be a line bundle on $X$ and $L$ the corresponding $\mathbb{G}_m$-torsor.
Then the stack $Y_L \to X$ carries an $\mathcal{L}^{\otimes m}$-twisted relative $n$-shifted symplectic structure of weight $m$.
\end{lem}
\begin{proof}
Let $\lambda: X \times B\BG_m \to B\BG_m$ be the map classifying the line bundle $\mathcal{L}\boxtimes
\mathcal{O}(-1)$.
We have a diagram with cartesian squares:
$$
\xymatrix{
Y_L \ar[r]\ar[d] & Y_L/\BG_m \ar[r]\ar[d] & X \times Y/\BG_m \ar[d] \\
X \ar[r] & X \times B\BG_m \ar[r]^{p_X \times \lambda} & X \times B\BG_m
}
$$
By Lemma \ref{p:rel forms}, we get an isomorphism
$$ \tilde{\mathcal{A}}^{2,cl}_{X\times B\BG_m}(Y_L/\BG_m) \simeq (p_X \times \lambda)^*(\tilde{\mathcal{A}}^{2,cl}_{X\times B\BG_m}(X \times Y/\BG_m)) .$$
In particular, the sheaf of weight $m$ relative closed 2-forms on $Y_L$ is given by
$$ \tilde{A}^{2,cl}_X(Y_L)(m) \simeq \mathcal{L}^{\otimes (- m)} \otimes \mathcal{A}^{2,cl}(Y)(m) .$$
By adjunction, we obtain a map
\begin{equation}\label{e:twisted forms}
\on{twist}_L: \mathcal{A}^{2,cl}(Y)(m) \to \Gamma(X, \tilde{\mathcal{A}}^{2,cl}_X(Y_L)(m)\otimes \mathcal{L}^{\otimes m}) .
\end{equation}
Thus, an $n$-shifted symplectic form of weight $m$ on $Y$ gives an $\mathcal{L}^{\otimes m}$-twisted relative closed 2-form of weight
$m$ on $Y_L$.  Moreover for a $\BG_m$-equivariant Lagrangian map $f: Z \to Y$, functoriality of
$\on{twist}_L$ induces a relative isotropic structure on $f_L: Z_L \to Y_L$.  Now, to see that the twisted relative closed 2-form on $Y_L$ is nondegenerate (resp. that $f_L$ is Lagrangian), it suffices to check this locally on $X$.  Thus, we can assume that $L$ is the trivial line bundle in which case the statement is manifest.
\end{proof}

The following is one of the main results of the paper.

\begin{thm}\label{c2} Let 
$Y$ be a  $\BG_m$-stack equipped with an $n$-shifted symplectic form
of weight $m\geq1$. Then, one has:

\vi The stack $\sec_X(Y_{K^{1/m}_X})$ has a natural
$(n-d)$-shifted  symplectic structure of weight $m$.

\vii For any Lagrangian structure $f:Z\to  Y$, of weight $\ell$,
 the  map $\sec_X(Z_{K^{1/\ell m}_X})\to \sec_X(Y_{K^{1/\m}_X})$, induced by $f$,
has a  natural Lagrangian structure of weight $\ell$. 
\end{thm}

\begin{proof}
Put $\mathcal{L} = K_X^{1/m}$, and let $L\to X$ be the corresponding $\BG_m$-torsor.  By Lemma \ref{l:twisted rel sympl}, we have that
$Y_L \to X$ has a $K_X$-twisted relative $n$-shifted symplectic structure of weight $m$.
By Proposition \ref{t:rel AKSZ} we obtain an $(n-d)$-shifted symplectic structure
on $\sec_X(Y_L)$, resp. Lagrangian structure,
on  $\sec_X(Z_L) \to \sec_X(Y_L)$.  Moreover, since the maps
\[ \sec_X(Y_L) \leftarrow \sec_X(Y_L) \times X \to Y_L \]
are $\BG_m$-equivariant,  the corresponding symplectic structure has weight $m$.
The required statements now follow from
an observation that, for any $\BG_m$-stack and a $\BG_m$-bundle $L\to X$, one has  natural isomorphisms of $\BG_m$-stacks
$\dis\ \sec_X(Y_{L^{\o m}})^{(m)} \simeq \sec_X(Y^{(m)}_L)$.
\end{proof}

We apply the above result to get a description of the symplectic structure on cotangent stacks to mapping stacks.

\begin{prop}\label{p:cotangent}
Let $Y = T^*[n] Z$ be the shifted cotangent stack with its $n$-shifted symplectic structure of weight 1.  In this case, there is a natural isomorphism of $(n-d)$-shifted symplectic stacks
$$ \sec_X(Y_{K_X}) \simeq T^*[n-d] \Map(X,Z). $$
\end{prop}
\begin{proof}
The symplectic form on $T^*[n] Z$ is given by the deRham differential of the canonical $n$-shifted 1-form on $T^*[n]Z$.  Therefore, it will suffice to construct an isomorphism of derived stacks $\sec_X(Y_{K_X}) \simeq T^*[n-d] \Map(X,Z)$ such that the transgression of the canonical 1-form is the canonical 1-form.

Recall that given a stack $W$ together with a quasi-coherent sheaf $\mathcal{E} \in \text{QCoh}(W)$, we can form the ``total space of $\mathcal{E}$'' as the stack $T(\mathcal{E})$ defined as follows.  A map from a test scheme $S$ to $T(\mathcal{E})$ is a map $f: S\to W$ together with a section of $f^*(\mathcal{E})$.  For instance, the stack $T^*[n] Z$ is the total space of the sheaf $\mathbb{L}_Z[n]$ on $Z$ and the canonical 1-form on $T^*[n]Z$ is given by the image of the section obtained from the identity map on $T^*[n]Z$ along
$$ p^* \mathbb{L}_Z[n] \to \mathbb{L}_{T^*[n]Z}[n], $$
where $p: T^*[n]Z\to Z$ is the projection map.

The projection map $p: T^*[n]Z\to Z$ gives a map $f:Y_{K_X} \to Z\times X$.  In fact, by construction, $Y_{K_X}$ is the total space of the sheaf $\mathbb{L}_Z[n] \boxtimes K_X$
on $Z\times X$.  In particular, we have a section of $\mathbb{L}_{Y_{K_X}/X}[n]\otimes K_X$ given by the image of the canonical section of $f^*(\mathbb{L}_Z[n] \boxtimes K_X)$ along the natural map
$$ f^*(\mathbb{L}_Z[n]\boxtimes K_X) \to \mathbb{L}_{Y_{K_X}/X}[n]\otimes K_X.$$
Moreover, the map $f$ induces the map
$$ g: \sec_X(Y_{K_X}) \to \Map(X,Z),$$
together with a section of $ev^* (\mathbb{L}_{Y_{K_X}/X}[n]\otimes K_X)$, where
$$ ev: \sec_X(Y_{K_X}) \times X \to \sec_X(Y_{K_X})$$
is the evaluation map.
Integrating along $X$, we obtain a section of $\pi_*(ev^* (\mathbb{L}_{Y_{K_X}/X}[n]\otimes K_X)) \simeq g^*(\mathbb{L}_{\Map(X,Z)}[n-d])$.  This gives the desired map of derived stacks
$$h: \sec_X(Y_{K_X}) \to T^*[n-d]\Map(X,Z),$$
which is easily seen to be an isomorphism.  Moreover, by construction, the pullback of the canonical 1-form on $T^*[n-d]\Map(X,Z)$ along $h$ is identified with the transgression of the canonical 1-form on $T^*[n]Z$, as desired.
\end{proof}

In addition to equivariant symplectic structures, we will also need to consider equivariant Calabi-Yau structures.

\begin{defn}
Let $S$ be a stack with a $\mathbb{G}_m$-action.  A $d$-Calabi-Yau structure of weight $m$ on $S$ is a map
$$ \Gamma(S,\mathcal{O}_X) \to \mathbb{C}[-d] $$
of weight $m$ satisfying the nondegeneracy condition of \cite[Definition 2.4]{PTVV}.
Equivalently, such a structure is given by a map of quasi-coherent sheaves on $B\mathbb{G}_m$
$$ \pi_*(\mathcal{O}_{S/\mathbb{G}_m}) \to \mathbb{C}(m)[-d], $$
where $\pi: S/\mathbb{G}_m \to B\mathbb{G}_m$ is the projection map.
\end{defn}

\begin{thm}\label{t:twisted CY}
Let $S$ be a $\mathbb{G}_m$-stack with a $d'$-Calabi-Yau structure of weight m.
Let $X$ be a smooth projective variety of dimension $d$ (or more generally, a derived stack with a twised orientation $K_X$ of degree $d$ as above) together with a choice of $K_X^{1/m}$. Then:

\vi
The stack $\tilde{X}:=X \underset{B\mathbb{G}_m}{\times} S/\mathbb{G}_m$ has a natural $(d+d')$ Calabi-Yau structure of weight $m$, 
where the map $X\to B\mathbb{G}_m$ classifies the line bundle $K_X^{1/m}$.

\vii
Given an $n$-shifted symplectic stack $Y$, there is a natural $\mathbb{G}_m$-equivaraint equivalence of $(n-d-d')$-shifted symplectic stacks of weight $m$
$$ \Map(\tilde{X}, Y) \simeq \sec_X(\Map(S, Y)_{K_X^{1/m}}). $$
\end{thm}

\begin{proof}
We have the Cartesian square of stacks
$$
\xymatrix{
\tilde{X}\ar[r]\ar[d] & S/\mathbb{G}_m \ar[d]^\pi\\
X \ar[r]^l & B\mathbb{G}_m
}
$$
Therefore, by base change, we have
$$ \Gamma(\tilde{X}, \mathcal{O}_X) \simeq \Gamma(X, l^* \pi_*(\mathcal{O}_{S/\mathbb{G}_m})) .$$
The desired Calabi-Yau structure on $\tilde{X}$ is then given as the composition of Calabi-Yau structures on $S$ and $X$:
$$ \Gamma(X, l^* \pi_*(\mathcal{O}_{S/\mathbb{G}_m})) \to \Gamma(X, l^*(\mathbb{C}(m)[d'])) \to \Gamma(X, K_X[d']) \to \mathbb{C}[d+d'].$$

\medskip

Now, we have isomorphisms
$$ \Map(\tilde{X}, Y) \simeq \sec_X(\Map_{/X}(\tilde{X}, Y\times X)) \simeq \sec_X(\Map(S, Y)_{K^{1/m}}), $$
which by construction of the Calabi-Yau structure on $\tilde{X}$ are compatible with the $(n-d-d')$-shifted symplectic structures of weight $m$. 
\end{proof}

\section{The case of $G$-bundles} 
For any stack $\Y$ 
and an integer $n$, the $n$-shifted  cotangent stack  $T^*[n]\Y$ comes equipped with 
a natural $n$-shifted symplectic form, see \cite[Proposition 1.21]{PTVV} and also \cite{Ca2}.
This 2-form has weight 1 with respect
to the $\BG_m$-action on $T^*[n]\Y$  by dilations along the fibers of the cotangent bundle.
The zero section $\Y\into T^*[n]\Y$ has a natural Lagrangian structure.

One has a canonical isomorphism  $\g^*/G=T^*[1]BG$,
which provides the stack
$\g^*/G$ with a natural $1$-shifted symplectic structure of weight 1.

In what follows, it will be convenient to have another description of this
$1$-shifted symplectic stack as a mapping stack.  Recall that an $Ad$-invaraint nondegenerate symmetric 
bilinear form $\kappa$ on $\g$ gives a 2-shifted symplectic structure on the stack $BG$.  Now, let $S=\widehat{B\mathbb{G}_a}$, the formal completion of $B\mathbb{G}_a$ at a point, with its natural $\mathbb{G}_m$ action.  We have that
$\Gamma(S, \mathcal{O}_S) \simeq \mathbb{C}[\epsilon]$, where $|\epsilon|=1$ and the map $\mathbb{C}[\epsilon]\to \mathbb{C}[-1],$ given by $\epsilon \mapsto 1$
gives $S$ a 1-Calabi-Yau structure of weight 1.  We then have:

\begin{lem}
There is a canonical isomorphism of $1$-shifted symplectic stacks of weight 1
$$ \Map(S, BG) \simeq T^*[1]BG .$$
\end{lem}
\begin{proof}
We have a $\mathbb{G}_m$-equivariant isomorphism of derived stacks $\Map(S, BG) \simeq T[-1]BG \simeq \mathfrak{g}/G$.  Recall that the 2-shifted symplectic structure on $BG$ is given by the image of an $Ad$-invariant symmetric bilinear form $\kappa$ on $\mathfrak{g}^*$ under the natural map
$$ \left(\oplus_{i\geq 0} Sym^{2+i}(\mathfrak{g}^*)[-2-2i]\right)^G \to \mathcal{A}^{2,cl}(BG). $$
Unraveling the definitions, we have that the composite map
$$\left(\oplus_{i\geq 0} Sym^{2+i}(\mathfrak{g}^*)[-2-2i]\right)^G \to \mathcal{A}^{2,cl}(BG) \to \mathcal{A}^{2, cl}(\mathfrak{g}/G)[-1]$$
factors through the map
$$
\left( \oplus_{p+q \geq l} \Omega^p(\mathfrak{g}) \otimes_{\mathbb{C}} Sym^q(\mathfrak{g}^*)[2-p-2q]) \right)^G \to \mathcal{A}^{2,cl}(\mathfrak{g}/G),
$$
where the differential in the complex on the left is given by the sum of the internal differential and the deRham differential on $\mathfrak{g}$.  Thus, we obtain that the only nonzero component of the 1-shifted symplectic structure on $\mathfrak{g}/G$ is given by the image of $\kappa$ along the map
$$ Sym^2(\mathfrak{g}^*) \to \Omega^1(\mathfrak{g}) \otimes \mathfrak{g}^* \simeq \mathfrak{g}^* \otimes \mathfrak{g}^* \otimes \mathcal{O}_{\mathfrak{g}}.$$ 
It follows that the $\mathbb{G}_m$ equivariant identification $\mathfrak{g}/G \simeq \mathfrak{g}^*/G$ induced by $\kappa$ upgrades to an isomorphism of $1$-shifted symplectic stacks of weight 1.
\end{proof}

\medskip

The map  $0/G\to \g^*/G$, induced
by the  imbedding $\{0\}\into \g^*$,  may be identified with the zero section $\imath: BG\to T^*[1]BG$.

Let   $M$ be a smooth symplectic variety equipped with a Hamiltonian $G$-action.
It was observed by Calaque  \cite{Ca}, that the map
 $M/G \to \g^*/G $,  induced by
the moment map  $\mu: M \to \g^* $, has a  natural Lagrangian structure.
Hence, from Theorem  \ref{c2} in the special case where $Y=\g^*/G$ and $n=1$ we deduce the following
result.

\begin{cor}\label{T^*Bun}  \vi For any  $m\geq1$, the stack $\sec_X(\g^*_{K^{1/\m}_X}/G)$ has  a canonical  $(1-d)$-shifted symplectic structure 
 structure of weight $m$.

\vii For a smooth symplectic $G\times \BG_m$-variety $M$ such that the
action of the group $G$ is hamiltonian and the symplectic 2-form has weight $\ell\ge1$ with respect
to the $\BG_m$-action,
the map $\sec_X(M_{K^{1/m\ell}_X})\to \sec_X(\g^*_{K^{1/\m}_X}/G)$, induced by the moment map
$M\to\g^*$, has a natural Lagrangian structure of weight $\ell$.
\end{cor}

We now specialize to the case where $\X=X$ is a smooth projective curve.
The stack   of Higgs bundles on $\X$ is defined
as $\higgs_G(\X):=\Map(\X_{Dol}, BG)$, where  $\X_{Dol}$ is the  Dolbeault stack, see \cite{PTVV}.  
Since $d=\dim \X=1$,  the stack $\higgs_G(\X)$ is  equipped with a
0-shifted symplectic  structure, by  \cite[Corollary 2.6(2)]{PTVV}.
 
\begin{lem}\label{l:tBG}
There are natural isomorphisms of 0-shifted symplectic stacks
$$ \Map(\X_{Dol}, BG) \simeq \sec_{\X}(\mathfrak{g}^*_K/G) \simeq T^*\Bun_G(\X) .$$
\end{lem}
\begin{proof}
By definition, $\X_{Dol}$ is identified with $X \underset{B\mathbb{G}_m}{\times} S,$ where the map $X \to B\mathbb{G}_m$
classifies $K_X$.  Moreover, by construction of the Calabi-Yau structure in Theorem \ref{t:twisted CY}(i), this isomorphism gives an isomorphism of 1-CY stacks.
The first, resp. second,  isomorphism of the lemma
then follows from Theorem \ref{t:twisted CY}(ii), resp. Proposition \ref{p:cotangent}.
\end{proof}
 
% The  stack $T^*\Bun_G(\X)$ has a canonical
%  0-shifted symplectic structure.
%
%
% It is immediate from definitions   that the natural isomorphism  $T^*\Bun_G(\X)\cong \sec_\X(\g^*_K/G)$ respects the symplectic structures.
%
% There is also  a natural isomorphism $\higgs_G(\X)\cong \sec_\X(\g^*_K/G)$.
%
%
% \begin{lem} The isomorphism $\higgs_G(\X)\cong \sec_\X(\g^*_K/G)$ respects
% the canonical $0$-shifted symplectic structure on each side.
% \end{lem}
% \begin{proof}
% Let $S:=\Spec(\C[\eps]/\eps^2) = B\BG_a$, where $\eps$ is placed in degree 1.
% The symmetric pairing on $\C[\eps]/\eps^2)$ given by $\eps\times\eps\mto 1$
% provides $S$ with a  1-CY structure.
%  The 2-shifted symplectic structure on $BG$
% yields, by Theorem \ref{thm}, a 1-shifted symplectic structure on $\Map(S, BG)$.
% The grading on  the algebra $\C[\eps]/\eps^2)$ gives a $\BG_m$-action on
% $S$ and an induced   $\BG_m$-action on  $\Map(S, BG)$.
% The  1-CY structure on $S$ has  weight $1$ with respect to the $\BG_m$-action.
% Therefore, the  1-shifted symplectic structure on $\Map(S, BG)$  has  weight $2$.
% We have a $\BG_m$-equivariant  isomorphism $T^*[1] G = \Map(S, BG)$
% that respects the  1-shifted symplectic structures.
%
%
% For a curve $\X$, one has natural isomorphisms
% $\X_{Dol} = \X/T_\X = \sec_\X (S_K)$, of $\BG_m$-stacks,
% where $K=K_\X$ and $T_\X$ is the tangent sheaf.
% Thus, we obtain
% \begin{multline*}
%   \higgs_G(\X) := \Map(\X_{Dol}, BG) = \Map(\sec_\X (S_K), BG) = \sec_\X(\Map(S, BG)_K) \\
% = \sec_\X(T^*[1]BG)_K =\sec_\X((\g^*/G)_K))=\sec_\X(\g^*_K/G).\qquad
% \qedhere
% \end{multline*}
% \end{proof}

Using the above lemma, from Corollary \ref{T^*Bun} we deduce

\begin{thm}\label{cor} Let  $M$ be a smooth symplectic $G\times \BG_m$-variety such that the
action of the group $G$ is hamiltonian and the symplectic 2-form has weight $\ell\geq 1$ with respect
to the $\BG_m$-action. Then, 
the map 
\beq{cor-sec}
\sec_\X(M_{K^{1/\ell}_\X}/G)\too \sec_\X(\g^*_{K_\X}/G)= T^*\Bun_G(\X),
\eeq
 induced by the moment map
$M\to\g^*$, has a natural Lagrangian structure of weight $\ell$.
\end{thm}

To complete the proof of
Theorem \ref{c1} one observes that, on the locus $T^*\Bun_G(\X)^{reg}$ where   $T^*\Bun_G(\X)$ is a smooth variety,
 the  0-shifted symplectic 2-form is nothing but the standard  symplectic 2-form  $\om$ on  $T^*\Bun_G(\X)^{reg}$ in the ordinary sense. 
Similarly,   
 if $\Lambda$ is  a smooth variety and
a map $f: \Lambda\to T^*\Bun_G(X)^{reg}$ has a Lagrangian structure then one has
$f^*\om=0$. Thus, Theorem \ref{c1} follows from Theorem \ref{cor}.

\section{Additional comments and speculations}
\subsection{A generalization of Gaiotto's  argument} 
In the  linear case, an  `infinite dimensional' approach to Theorem \ref{c1} 
is explained in  \cite{Ga}. Gaiotto's approach is
based on a standard differential geometric interpretation of
$\Bun_G(\X)$ as a quotient of an  infinite dimensional space of $\bar\partial$-connections by a gauge group.
It was suggested  to us by Gaiotto that the argument in  \cite{Ga}  can be adapted to the more general, nonlinear setting of
Theorem \ref{c1} as follows. Below, we assume that $\ell=2$, for simplicity.

Fix   a  principal $C^\infty$-bundle $P\xrightarrow{G}\X$ and
let $\textit{Conn}_{\bar\partial}(P)$ be (an infinite dimensional) space
of  $\bar\partial$-connections on $P$. 
Further, let $\sec_{\X,C^\infty}(M_{K_ X^{1/2}}\times_G P)$
be  (an infinite dimensional) space of $C^\infty$-sections of an associated bundle $M_{K_ X^{1/2}}\times_G P\to\X$.
Let
$z\in \sec_{\X,C^\infty}(M_{K_ X^{1/2}}\times_G P)$ be such a section and
$ A\in \textit{Conn}_{\bar\partial}(P)$
a $\bar\partial$-connection.
 Then $\nabla_{\!_{ A}}z$, a covariant derivative of $z$ with respect to $  A$,
is a $C^\infty$-section of $z^*T_M\o K_\X^{1/2}\o\Om^{0,1}_\X$, where $T_M$ stands for the holomorphic tangent
sheaf on $M$ and $\Om^{p,q}_\X$ is the sheaf of $C^\infty$ differential forms on $\X$ of type $(p,q)$.
Further,  let $\la_M=i_{\eu_M}\om_M$,
where $\om_M$ is the (holomorphic) symplectic form on $M$ and $\eu_M$ is the Euler field
that generates the $\BG_m$-action on $M$.
Thus, $z^*\la_M$ is a $C^\infty$-section of $z^*T^*_M\o K_\X^{1/2}$. Using
the  canonical pairing $\langle -,-\rangle$ of holomorphic vector fields and  holomorphic 1-forms on $M$,
we obtain a $C^\infty$-section $\langle \nabla_{\!_{ A}}z,z^*\la_M\rangle$
of the sheaf $K_\X^{1/2}\o\Om^{0,1}_\X\o K_\X^{1/2}=K_\X\o \Om^{0,1}_\X=\Om^{1,1}_\X$.

In the above setting, the role of  the potential   from \cite[formula (2.3)]{Ga} is played by
a function on $\sec_{\X,C^\infty}(M_{K_ X^{1/2}}\times_G P)\,\times \,\textit{Conn}_{\bar\partial}(P)$ defined by
the formula
\beq{W}W(z, A)=\int_\X \langle \nabla_{\!_{ A}}z,z^*\la_M\rangle.
\eeq

To prove  that the map $\mu_{\sec}$ in Theorem \ref{c1} is Lagrangian
we show, by a calculation similar to the one in \cite[Appendix A]{Ga}, that \eqref{W} is a generating function
(aka `Lagrange multiplier') for $\sec_\X( M_{K_ X^{1/2}}/G)$.

To this end, observe that an infinitesimal variation of $z$ is given by a section $\overset{.}{z}$ of ~$z^*T_M\o K_ X^{1/2}$.
The corresponding variation of the $(1,1)$-form $\langle \nabla_{\!_{ A}}z,z^*\la_M\rangle$
reads
$$\frac{\delta \langle \nabla_{\!_{ A}}z,z^*\la_M\rangle}{\delta z}(\overset{.}{z})=\bar\partial\langle \nabla_{\!_{ A}}\overset{.}{z},z^*\la_M\rangle+
(d\la_M)(\nabla_{\!_{ A}}z,\overset{.}{z}),
$$
where the operator  $\bar\partial$ that appears  in the first summand on the right
is  the Dolbeault differential $\bar\partial: \Om^{1,0}_\X\to \Om^{1,1}_\X$.
Using that $d\la_M=\om_M$ and that, on $\Om^{1,0}_\X$, one has $\bar\partial=d$, we find that the 
variation of \eqref{W} equals
$$\frac{\delta W}{\delta z}(\overset{.}{z})=\int_\X d\langle \nabla_{\!_{ A}}\overset{.}{z},z^*\la_M\rangle+
\int_\X \om_M(\nabla_{\!_{ A}}z,\overset{.}{z}).
$$
The first summand on the right vanishes by Stokes' theorem. Hence, the form $\om_M$ being nondegerate,  
we deduce that the equation $\frac{\delta W}{\delta z}(\overset{.}{z})=0$ holds for all $\overset{.}{z}$
if and only if $\nabla_{\!_{ A}}z=0$, that is, if and only if the section $z$ is holomorphic with respect to the complex
structure on $M_{K_ X^{1/2}}\times_G P$ determined by the $\bar\partial$-connection $A$.

Next, let $\overset{\ .}{A}$ be an infinitesimal variation  of $A$.
Then, 
it is easy to check that $\frac{\delta W}{\delta A}(\overset{\ .}{A})=\mu_\sec(z,A)(\overset{\ .}{A})$,
proving that $W$ is a generating function for $\sec_\X( M_{K_ X^{1/2}}/G)$.

\begin{rem}
Let $\eu$,  resp.  $\eu_\sec$, be the Euler vector field on 
$T^*\Bun_G(\X)$, resp. $\sec_\X(M_{K_ X^{1/2}}/G)$, that generates the $\BG_m$-action. 
Recall that  $\om=d\la$ where $\la=i_{\eu}\om$ is the Liouville 1-form  on $T^*\Bun_G(\X)$.
The map $\mu_\sec$ in \eqref{mumap} being of weight $\ell$, one finds:
$$\mu_{\sec}^*(\la)=\mu_{\sec}^*(i_{\eu}\om)=\mbox{$\frac{1}{\ell}$}\cdot i_{\eu_{\sec}}\mu^*(\om).$$
It follows, as has been observed by Hitchin \cite{Hi}, that Theorem \ref{cor}  is equivalent to the equation
 ~$\mu_{\sec}^*(\la)=0$.
\end{rem}

\subsection{Relation to the global nilpotent cone}
Let  $B$ be a Borel subgroup of $G$, so $G/B$ is the flag variety.
The symplectic form on $T^*(G/B)$ has weight 1 and the moment map
$\mu: T^*(G/B)\to\g^*$ is  the Springer resolution $T^*(G/B)\to\N$, where
$\N\sset \g^*$ is the  nilpotent cone. 
The stack $\sec_\X((\N/G)_{K_\X})$ can be identified with $\N_\X$, the 
{\em global nilpotent cone} in $T^*\Bun_G(\X)$.
Further, the stack 
  $\sec_\X((T^*(G/B)_{K_\X}/G)$  can be identified with $T^*\Bun_B(\X)$. Explicitly,
writing $\n$ for the nilradical of $\Lie B$,  the stack  $T^*\Bun_B(\X)$ classifies triples $(P,\sigma,\phi)$, where $P$ is a $G$-bundle on $\X$,
$\sigma: \X\to P/B$ is a section, i.e. a reduction of $P$ to a $B$-bundle, and $\phi: P\times_B\n\to (P\times_B\n)\o K_\X$
is a Higgs field. Assume that the genus of the curve $\X$ is greater than 1.
Then, the derived  stacks $T^*\Bun_G(\X)$ and $T^*\Bun_B(\X)$ are concentrated in homological degree 0, i.e. they are
actually non-derived stacks. 
The stack  $\N_\X$ is not concentrated in homological degree 0,
and one can consider  $\N_\X^{\textrm{\tiny classical}}$, its  non-derived counterpart,
which is an ordinary substack
of  $T^*\Bun_G(\X)$.

The  map $(P,\sigma,\phi)\mto (P,\phi)$, that forgets  reduction of the structure group,  may be identified
with the composition
\beq{mn}\mu_\sec: T^*\Bun_B(\X)\
\xrightarrow{\pi_1}\ \N_\X\
\xrightarrow{\pi_2}\  T^*\Bun_G(\X).
\eeq
The map $\mu_\sec$  has  a Lagrangian structure by  Corollary  \ref{T^*Bun}.
One can show that the map
$\pi_2$ has a natural coisotropic structure  in the sense of \cite{MS}.
However, this coisotropic structure is easily seen to be {\em not} Lagrangian.

On the other hand, it was shown in \cite{Gi} that, for any field extension $K/k$,
the  map  $\pi_1^{\textrm{\tiny classical}}:\ T^*\Bun_B(\X)(\Spec K)\to \N_\X^{\textrm{\tiny classical}}(\Spec K)$, of $K$-points
of the corresponding {\em non-derived} stacks, is surjective.
This result  was used in \cite{Gi} to prove 
 that $\N_\X^{\textrm{\tiny classical}}$ is  (as opposed to its derived analogue) a Lagrangian substack of  $T^*\Bun_G(\X)$
in the sense  explained in {\em loc cit}. 
\bigskip

More generally, let $\wt Y\to Y$ be a $(G\times\BG_m)$-equivariant symplectic resolution such that 
$Y$ is affine,  the $\BG_m$-action on $Y$ is a contraction to a unique 
$\BG_m$-fixed point and, moreover, the
symplectic form on $\wt Y$ has weight $m\geq 1$.
Then,  we have  $k[\wt Y]=k[Y]$, so the  Poisson bracket on the algebra $k[\wt Y]$ 
provides $Y$ with a $(G\times\BG_m)$-equivariant Poisson structure. Also, the moment map $\wt Y/G\to\g^*/G$ factors through $Y/G$.
Therefore, there is a chain of induced maps
$\dis\sec_\X((\wt Y/G)_{K^{1/m}_\X})\xrightarrow{\pi_1}\sec_\X(Y_{K^{1/m}_\X}/G)\xrightarrow{\pi_2} T^*\Bun_G(\X)$
such that $\pi_2\ccirc\pi_1=\mu_\sec$. The map $\mu_\sec$ has  a Lagrangian structure, by Theorem \ref{cor}.
Again,  one can show that the  map $\pi_2: \sec_\X(Y_{K^{1/m}_\X}/G)\xrightarrow{\pi_2} T^*\Bun_G(\X)$
has a natural coisotropic structure.

\begin{quest} Is $\sec_\X(Y_{K^{1/m}_\X}/G)^{\textrm{\tiny classical}}$,
a non-derived counterpart of $\sec_\X(Y_{K^{1/m}_\X}/G)$, isotropic in the sense of
\cite{Gi}, specifically, is it possible to partition $\sec_\X(Y_{K^{1/m}_\X}/G)^{\textrm{\tiny classical}}$ as a disjoint union
of   substacks such that the pull-back of the symplectic 2-form on $T^*\Bun_G(\X)$ to each of these substacks vanishes ?
\end{quest}

\subsection{Hamiltonian reduction}
Let $M$ be  a stack equipped with a  0-shifted symplectic structure
and with a Hamiltonian $G$-action
with moment map $\mu$. The stack $\mu\inv(0)/G$,
 a stacky Hamiltonian reduction of $M$, comes equipped with a canonical  0-shifted symplectic structure.
On the other hand, 
let $\Lambda_1=0/G\to \g^*/G$ be the map induced by the imbedding
$\{0\}\into \g^*$ and $\Lambda_2=M/G\to \g^*/G$ be the map induced by $\mu$.
One has a natural isomorphism  
\beq{ham-red}\Lambda_1 \times_{\g^*/G}\Lambda_2=0/G\,\times_{\g^*/G}\, M/G\cong \mu\inv(0)/G.
\eeq
Recall that the stack $\g^*/G$ has the canonical $1$-shifted symplectic structure
and each of the two maps $\Lambda_i\to \g^*/G,\ i=1,2$, has a Lagrangian structure, cf. \S4.
Further, according to \cite[Theorem 0.5]{PTVV},
for any stack $\Y$ equipped with an $n$-shifted symplectic structure
and a pair $\Lambda_i\to \Y,\ i=1,2$, of Lagrangian structures,
the stack $\Lambda_1\times_\Y\Lambda_2$ has a natural $(n-1)$-shifted symplectic structure.
Therefore, the stack 
$0/G\,\times_{\g^*/G}\, M/G$ comes equipped with a  0-shifted symplectic structure.
It was shown by Calaque \cite{Ca} that the isomorphism in \eqref{ham-red}
respects the   0-shifted symplectic structures.

Next, we   fix a smooth projective curve $\X$ and let  $K=K_\X$. The stack
$\g^*_K/G=T^*\Bun_G(\X)$, a global counterpart of   $\g^*/G$,
has the  0-shifted symplectic structure of weight 1. Also,
the   Lagrangian structure on the map $0/G\to \g^*/G$
induces, for any $\ell$, a weight $\ell$  Lagrangian structure  $\sec_\X ((0/G)_{K^{1/\ell}})\to \sec_\X ((\g^*/G)_K)$.
The latter  Lagrangian structure corresponds,
via the isomorphisms $T^*\Bun_G(\X)\cong \g^*_K/G$ and $Bun_G(\X)\cong  \sec_\X ((0/G)_K)$,
to  an obvious  Lagrangian structure on 
the zero section $\Bun_G(\X)\to T^*\Bun_G(\X)$.
\big(We have used here that for any variety $\Y$ equipped with a trivial $\BG_m$-action and any
$\BG_m$-bundle $L$ on $\X$, one has $\sec_X(\Y_L)=\Map(\X,\Y)$, in particular,
we have $\sec_\X ((0/G)_K)=\Map(\X,BG)=Bun_G(\X)$.\big)

Now,
let $M$ be a symplectic manifold equiped with a $(G\times\BG_m)$-action
such that the symplectic 2-form has weight  $\ell\geq1$ and
the $G$-action is Hamiltonian. 
One has canonical isomorphisms
\begin{align}
\sec_\X ((0/G)_{K^{1/\ell}})\ &\times_{T^*\Bun_G(\X)}\ \sec_\X ((M/G)_{K^{1/\ell}})\label{iso}\\
&\cong\ \sec_\X \big((0/G)_{K^{1/\ell}}\,\times_{\g^*_{K}/G}\, (M/G)_{K^{1/\ell}}\big)\ \cong\ \sec_\X((\mu\inv(0)/G)_{K^{1/\ell}}).\nonumber
\end{align}
Here,  the fiber product on the left involves the map \eqref{cor-sec}, which
has a weight $\ell$ Lagrangian structure, by Theorem \ref{cor}.
Thus, according to \cite[Theorem 0.5]{PTVV},
the fiber product of Lagrangians 
on the left of \eqref{iso}  has a $(-1)$-shifted symplectic structure. On the other hand,
the 
0-shifted symplectic structure
on $\mu\inv(0)/G$ induces, by Theorem \ref{c2}(i), a  $(-1)$-shifted symplectic structure 
of weight $\ell$ on
$\sec_\X((\mu\inv(0)/G)_{K^{1/\ell}})$, the stack on the right  of \eqref{iso}.
One can check that the composite isomorphism
in \eqref{iso} respects the $(-1)$-shifted symplectic structures described above.

Let $\XX$ be a stack and assume there is  a line bundle 
$K^{1/2}_\XX$, a square root of the dualizing complex of $\XX$.  
In \cite{Pr}, Pridham shows that  an $(-1)$-shifted symplectic structure on  $\XX$ gives rise
to a canonical self-dual quantization of $K^{1/2}_\XX$. Moreover, associated with that quantization,
there is a constructible complex on $\XX$, of  vanishing cycles. Therefore, one might expect that,
in the setting of the previous paragraph, 
the stack $\sec_\X((\mu\inv(0)/G)_{K^{1/\ell}})$ comes equipped (perhaps,
under some additional assumptions) with a natural constructible complex 
 of  vanishing cycles. 

The linear case, where  $\ell=2$ and $M$ is a linear symplectic representation of $G$,
has been considered in the physics literature in the framework of Coulomb branches for 3-dimensional gauge theories,
cf. \cite{Ga} and references therein. The special case where $M=E\oplus E^*$ is a direct sum
of a pair of dual representations of  $G$ is  simpler than the general case. In that case, 
the geometry of $\sec_\X((\mu\inv(0)/G)_{K^{1/2}})$ can be reduced, in a sense,
to the geometry of $\sec_\X(E_{K^{1/2}})$. Such a  reduction allows to avoid the use of vanishing cycles. 
A mathematical  theory of Coulomb branches in the case $M=E\oplus E^*$ was developed
by H. ~Nakajima \cite{Na}, cf. also \cite{BFN}.

{\small
\bibliographystyle{plain}

\begin{thebibliography}{PTVVV}
 \bibitem[BFN]{BFN} A.  Braverman, M.  Finkelberg, H. Nakajima, {\em
    Towards a mathematical definition of Coulomb branches of 3-dimensional N=4 gauge theories, II.}
arXiv:1601.03586.

\bibitem[Ca]{Ca} D. Calaque, {\em Lagrangian structures on mapping stacks and semi-classical TFTs.}
 Stacks and categories in geometry, topology, and algebra, 1-23, Contemp. Math., 643, Amer. Math. Soc., Providence, RI, 2015.
  arXiv:1306.3235.
\bibitem[Ca2]{Ca2} \bysame, {\em Shifted cotangent stacks are shifted symplectic.}
arxiv:1612.08101.
\bibitem[CPTVV]{CPTVV} \bysame,  T. Pantev, B. To\"en, G. Vezzosi, {\em Shifted Poisson structures and deformation quantization}.
arXiv:1506.03699.
\bibitem[GR]{GR} D. Gaitsgory, N. Rozenblyum, {\em A study in derived algebraic geometry}.
Mathematical Surveys and Monographs, Vol. 221.

%\begin{verbatim}http://www.math.harvard.edu/˜gaitsgde/GL/
%\end{verbatim}

\bibitem[Ga]{Ga} D. Gaiotto, {\em 
    S-duality of boundary conditions and the Geometric Langlands program.}
 arXiv:1609.09030.
\bibitem[Gi]{Gi} V. Ginzburg, {\em The global nilpotent variety is Lagrangian.}
 Duke Math. J. 109 (2001),  511-519.

 %\bibitem[HV]{HV} T. Hausel, F. Rodriguez Villegas, {\em Cohomology of large semiprojective hyperkähler varieties.} Ast\'erisque No. 370 (2015), 113-156. 

\bibitem[Hi]{Hi} N. Hitchin, {\em
Spinors, Lagrangians and rank 2 Higgs bundles.}
arXiv:1605.06385.

\bibitem[MS]{MS} V. Melani, P. Safronov, {\em
    Derived coisotropic structures.}
 arXiv:1608.01482.

\bibitem[Na]{Na}  H. Nakajima, {\em
    Towards a mathematical definition of Coulomb branches of 3-dimensional N=4 gauge theories, I.}
 arXiv:1503.03676.

\bibitem[PTVV]{PTVV} T.
Pantev, B. To\"en, M. Vaqui\'e, G. Vezzosi, {\em 
Shifted symplectic structures.}
Publ. Math. Inst. Hautes \'Etudes Sci. 117 (2013), 271-328. 
\bibitem[Pr]{Pr} J. P. Pridham, {\em
 Deformation quantisation for $(-1)$-shifted symplectic structures and vanishing cycles.}$\qquad\qquad$\hfill
arXiv:1508.07936.

\bibitem[Sa]{Sa} P. Safronov, {\em
    Quasi-Hamiltonian reduction via classical Chern-Simons theory.}
 Adv. Math. 287 (2016), 733-773. arXiv:1311.6429.
\bibitem[Sp]{derived2} T. Spaide, 
{\em
    Shifted Symplectic and Poisson Structures on Spaces of Framed Maps.}
arXiv:1607.03807.

\end{thebibliography}

}
\end{document}